\documentclass[11pt]{amsproc}
\usepackage[margin=1in]{geometry}
\usepackage{setspace,fullpage}
\usepackage{graphicx}
\usepackage[nice]{nicefrac}
\usepackage{amssymb}
\usepackage{multirow}
\usepackage{array}
\usepackage{amsmath,amsthm,amscd,amssymb, mathrsfs}
\usepackage{latexsym}
\usepackage[colorlinks=true,citecolor=blue,linkcolor=blue,urlcolor=blue]{hyperref}

\numberwithin{equation}{section}

\theoremstyle{plain}
\newtheorem{theorem}{Theorem}[section]
\newtheorem{lemma}[theorem]{Lemma}
\newtheorem{corollary}[theorem]{Corollary}
\newtheorem{proposition}[theorem]{Proposition}

\theoremstyle{definition}

\theoremstyle{remark}

\newcommand{\F}{\mathbb F}
\newcommand{\C}{\mathbb C}
\newcommand{\T}{\mathbb T}
\newcommand{\Z}{\mathbb Z}
\newcommand{\Rcal}{\mathcal R}
\newcommand{\Tcal}{\mathcal T}
\newcommand{\Bcal}{\mathcal B}
\newcommand{\Qcal}{\mathcal Q}
\newcommand{\eps}{\varepsilon}
\newcommand{\supp}{\operatorname{supp}}

\title{A bilinear approach to the finite field restriction problem, II}
\author{Mark Lewko}
\address{Enfield, New Hampshire USA}
\email{mlewko@gmail.com}
\hypersetup{pdftitle={A bilinear approach to the finite field restriction problem, II},pdfauthor={Mark Lewko}}

\begin{document}

\begin{abstract}
Let $P_3$ denote the three-dimensional paraboloid over a finite field of prime order in which $-1$ is not a square. We prove that the Fourier extension operator associated with $P_3$ maps $L^2$ to $L^r$ for $r>\frac{176}{51}=3.45098\ldots$. The argument combines the author's bilinear approach to the problem with point-line incidence estimates. We also establish the sharp endpoint $L^2\to L^{8/3}$ estimate for the extension operator associated with the paraboloid $P_6$ in six dimensions, previously known only away from the endpoint. Finally we observe that the endpoint restriction conjecture for $P_3$ in finite fields implies that the integer lattice points on the three-dimensional Euclidean paraboloid are a $\Lambda(3)$ set.
\end{abstract}

\maketitle

\section{Introduction}

The Fourier restriction problem over finite fields was introduced by Mockenhaupt and Tao \cite{MT} and has since been studied extensively, with connections to additive combinatorics, incidence geometry, Kakeya phenomena, the theory of quadratic forms, and explicit extractor constructions, among other topics. We refer the reader to \cite{LewK} and \cite{MT} for an introduction to the problem, and to \cite{IK09,IKL,LewRect,LewN,RS,SD} for some of the more recent developments.

Let $F$ be a finite field of odd characteristic, with $q=|F|$, and let
\[
P_d=\{(\underline x,\underline x\cdot \underline x):\underline x\in F^{d-1}\}\subset F^d
\]
denote the $d$-dimensional paraboloid, endowed with the normalized counting measure $d\sigma$ that assigns mass $|P_d|^{-1}=q^{-(d-1)}$ to each point. For $f:P_d\to\C$ the Fourier extension operator is given by
\[
(f d\sigma)^\vee(x)=q^{-(d-1)}\sum_{\xi\in P_d}f(\xi)\,e(x\cdot \xi),\qquad x\in F^d,
\]
where $e$ is a fixed nontrivial additive character of $F$. This work will focus entirely on the $L^2$ theory.  We write $R^*(2\to r)$ for the best constant in
\[
\|(f d\sigma)^\vee\|_{L^r(F^d)}\le R^*(2\to r)\,\|f\|_{L^2(P_d,d\sigma)},
\]
with the implied constant required to be independent of $q$. By duality this is equivalent to the inequality $\|\widehat g\|_{L^2(P_d,d\sigma)}\lesssim\|g\|_{L^{r'}(F^d)}$, with $r'=\frac{r}{r-1}$.

We begin with the three-dimensional paraboloid. The Stein--Tomas method gives $R^*(2\to r)\lesssim1$ for $r\ge4$, which is the optimal $L^2$ range over an arbitrary finite field. When $-1$ is not a square in $F$ the surface contains no lines, and one expects a larger range; in this setting Mockenhaupt and Tao \cite{MT} proved $R^*(2\to r)\lesssim1$ for $r>\frac{18}{5}=3.6$ and conjectured that the inequality holds for every $r\ge3$. Their range was improved to $r>\frac{18}{5}-\delta$, for some small $\delta>0$, in arbitrary fields of odd characteristic in which $-1$ is not a square, and to $r\ge\frac{745}{207}\approx3.599$ in prime fields, by the author \cite{LewN}. In prime fields the range was subsequently lowered to $r>\frac{68}{19}\approx3.579$ by Stevens and de Zeeuw \cite{SD}, to $r\ge\frac{32}{9}\approx3.555$ by Rudnev and Shkredov \cite{RS}, and to $r>\frac{188}{53}\approx3.547$ by the author \cite{LewRect}. In arbitrary fields of odd characteristic in which $-1$ is not a square, the first paper in this series \cite{LewBilinear} gives $r>\frac{32}{9}=3.555\ldots$.\footnote{The published version of \cite{LewBilinear} stated the larger range $r>\frac{24}{7}$, owing to an error in the final dominance calculation, where a term of the form $q^{-3/16}|G|^{13/16}$ was taken to be dominated by $|G|^{17/24}$ throughout the relevant range, whereas this holds only for $|G|\le q^{9/5}$; the worst case is at $|G|\sim q^2$ and gives $r>\frac{32}{9}$. This is discussed in an erratum accompanying the revised version of \cite{LewBilinear} on the arXiv, whose estimates we use here. That paper also remarked that one of the terms could be lowered using incidence estimates but, believing that term not to be dominant, stated this would not improve the main result; it is in fact the dominant term, and improving the term and the restriction estimate is the purpose of Section \ref{sec:prime}.}

These improvements have followed two distinct paths. Beginning from the Mockenhaupt--Tao exponent $r>\frac{18}{5}$, nearly all of the progress has come from sharpening the additive-energy input in the Mockenhaupt--Tao argument. That input is bounded using point-line incidence estimates, and ultimately Rudnev's point-plane incidence theorem \cite{R}, which in turn rests on a result of Koll\'ar \cite{Kol} in algebraic geometry. Incidence bounds of this type are sensitive to the presence of subfields, so the resulting restriction estimates have been confined to prime fields; see the survey \cite{Rsur}. The second and more recent idea is the bilinear approach of \cite{LewBilinear}, which replaces the Mockenhaupt--Tao argument by an estimate for a related bilinear operator, of which the Mockenhaupt--Tao argument is essentially the diagonal. This approach has the advantage of working in arbitrary fields of odd characteristic. Here we combine the two approaches, to prove the following:

\begin{theorem}\label{thm:prime-main}
Let $F=\F_p$ be a field of prime order in which $-1$ is not a square, and let $P_3=\{(\underline x,\underline x\cdot \underline x):\underline x\in F^{2}\}$. Then $R^*(2\to r)\lesssim_r1$ for every
\[
r>\frac{176}{51}=3.45098\ldots .
\]
\end{theorem}

\noindent\emph{Note added.} Since this paper was written, the author has obtained a sharp Szemer\'edi--Trotter theorem in \cite{LewST}. Its application to the three-dimensional paraboloid improves Theorem \ref{thm:prime-main} to $R^*(2\to r)\lesssim_r1$ for every $r>10/3$, under the same hypotheses.

Our second main result is the sharp endpoint estimate for the paraboloid in six dimensions. In \cite{IKL}, Iosevich, Koh and the author established the optimal $L^2$ restriction estimate for $P_d$ in every even dimension $d\ge8$, and in dimension six proved $R^*\big(2\to\frac83+\eps\big)\lesssim1$ for every $\eps>0$, which is sharp apart from the endpoint. Here we obtain the endpoint.

\begin{theorem}\label{thm:six-main}
Let $F$ be a finite field of odd characteristic and let $P_6=\{(\xi,\xi\cdot \xi):\xi\in F^5\}\subset F^6$. Then
\[
R^*\Big(2\to\frac83\Big)\lesssim1,\qquad\text{equivalently}\qquad
\|(f d\sigma)^\vee\|_{L^{8/3}(F^6)}\lesssim\|f\|_{L^2(P_6,d\sigma)}.
\]
\end{theorem}

In the final section we record a transference observation. The conjectured endpoint $L^2\to L^3$ restriction estimates for the paraboloid and for spheres containing no affine lines in $F^3$ would imply that the integer paraboloid and the lattice spheres in $\Z^3$ are $\Lambda(3)$ sets, with a constant independent of the size of the support. The precise uniformity hypothesis for spheres is given in Section \ref{sec:lambda}. This implication, with no power loss in the size of the support on the right-hand side, is perhaps some indication of the difficulty of the finite field endpoint problem. In particular, it would make the full integer paraboloid an explicit $\Lambda(3)$ set that is not a $\Lambda(4)$ set, the latter fact following from Bourgain's logarithmic lower bound \cite{BouLattice}.

\section{Background and notation}\label{sec:background}

We use counting measure on $F^d$ and normalized counting measure on the paraboloid, so for a set $S$ we write $1_S$ for its indicator and $|S|$ for its cardinality. For $g:F^d\to\C$ we set $\widehat g(\xi)=\sum_{x\in F^d}g(x)e(-x\cdot\xi)$. We write $X\lesssim Y$ to mean $X\le CY$ for a constant $C$ independent of $q$, and $X\lesssim_\eps Y$ when the constant is allowed to depend on a parameter $\eps$; the relation $X\sim Y$ means $X\lesssim Y$ and $Y\lesssim X$. For a function $g$ the notation $g(x)\sim\lambda$ means $\lambda/2\le|g(x)|\le2\lambda$. We collect the various estimates we will need from the papers \cite{IKL,LewBilinear,LewK,LewRect,MT}; we state them without proof and refer to those papers for the details.

An estimate of the form $\|(f d\sigma)^\vee\|_{L^r(F^d)}\lesssim\|f\|_{L^2(P_d,d\sigma)}$ is dual to
\begin{equation}\label{eq:dual}
\|\widehat g\|_{L^2(P_d,d\sigma)}\lesssim\|g\|_{L^{r'}(F^d)},\qquad r'=\tfrac{r}{r-1},
\end{equation}
and one has the identity $\|\widehat g\|_{L^2(P_d,d\sigma)}^2=\langle g,g*(d\sigma)^\vee\rangle_{F^d}$, the inner product being taken with respect to counting measure. Here $(d\sigma)^\vee=\delta_0+K$, where for $x=(\underline x,x_d)\in F^{d-1}\times F$ with $x\ne0$ the kernel is the Gauss sum
\[
K(x)=q^{-(d-1)}\sum_{\underline\xi\in F^{d-1}}e\big(\underline x\cdot\underline\xi+x_d\,\underline\xi\cdot\underline\xi\big).
\]
Writing $G(a)=\sum_{t\in F}e(at^2)$ for the one-dimensional Gauss sum, with $|G(a)|=q^{1/2}$ for $a\ne0$, completing the square evaluates this as
\[
K(x)=q^{-(d-1)}G(x_d)^{d-1}\,e\Big(-\frac{\underline x\cdot\underline x}{4x_d}\Big)\quad(x_d\ne0),\qquad K(x)=0\quad(x_d=0,\ x\ne0)
\]
(see (18) of \cite{MT}). In particular $|K(x)|\le q^{-(d-1)/2}$ for $x\ne0$, which equals $q^{-1}$ in the three-dimensional case $d=3$. The decomposition yields
\begin{equation}\label{eq:steintomas}
\|\widehat g\|_{L^2(P_d,d\sigma)}\le\|g\|_{L^2(F^d)}+q^{-(d-1)/4}\|g\|_{L^1(F^d)},
\end{equation}
which in dimension $d=3$ gives the bound with coefficient $q^{-1/2}$. The usual Stein--Tomas argument combines the Fourier decay of $K$ with Parseval to obtain the uniform extension estimate.

The kernel $K$ also relates the fourth moment of a convolution to the restriction of that convolution to horizontal slices. For $h:F^d\to\C$ and $z\in F$ write $h_z(\eta)=h(\eta,z)$ for the slice of $h$ at height $z$, a function on $F^{d-1}$, and let $\widetilde h_z$ be its lift to the paraboloid, the function on $P_d$ defined by $\widetilde h_z(\eta,\eta\cdot\eta)=h_z(\eta)$.

\begin{lemma}[{\cite[proof of Lemma 2.1]{IKL}}]\label{lem:slice}
For every $h:F^d\to\C$,
\begin{equation}\label{eq:slice}
\|h*K\|_{L^4(F^d)}\lesssim q^{(d-1)/2}\sum_{z\in F}\big\|(\widetilde h_zd\sigma)^\vee\big\|_{L^4(F^d)}.
\end{equation}
\end{lemma}

This reduces the fourth moment of $h*K$ to the extension operators of the planar slices $\widetilde h_z$, each of which is estimated by the additive energy of its support through Plancherel.

We use two standard reductions, both recalled in \cite{LewBilinear,LewK}. The first is dyadic pigeonholing, which reduces matters, at the cost of a factor $q^\eps$, to functions $g$ with $g\sim1$ on their support $G$. The second is the following epsilon-removal lemma.

\begin{lemma}[{\cite[Lemma 16]{LewK}}]\label{lem:eps-removal}
Let $s_0>1$, and suppose that for every $\eps>0$,
\begin{equation}\label{eq:restricted}
\|\widehat g\|_{L^2(P,d\sigma)}\lesssim_\eps q^{\eps}|G|^{1/s_0}
\end{equation}
holds for every $g$ with $g\sim1$ on its support $G$. Then
\begin{equation}\label{eq:strong-type}
\|\widehat g\|_{L^2(P,d\sigma)}\lesssim_s\|g\|_{L^s(F^d)}\qquad\text{for every }1\le s<s_0.
\end{equation}
\end{lemma}

By \eqref{eq:dual}, the bound \eqref{eq:strong-type} is the extension estimate $R^*(2\to r)\lesssim_r1$ for every $r>s_0'$, where $s_0'=\frac{s_0}{s_0-1}$. It therefore suffices to prove the restricted estimate \eqref{eq:restricted} for $g\sim1_G$.

For the combinatorial input we record three quantities. For a finite subset $E$ of an abelian group, its additive energy is
\begin{equation}\label{eq:energy-def}
\Lambda(E)=|\{(a,b,c,d)\in E^4:a+b=c+d\}|.
\end{equation}
For $A\subset F^2$ we call a triple $(x_0,x_1,x_2)$ a corner if $(x_1-x_0)\cdot(x_2-x_1)=0$, and a quadruple $(x_0,x_1,x_2,x_3)$ a rectangle if each cyclically consecutive triple is a corner; $\Rcal(A)$ denotes the number of rectangles in $A$. For $A,B\subset F^2$ we call $(x_1,x_2,y_1,y_2)$, with $x_1,x_2\in A$ and $y_1,y_2\in B$, a trapezoid if $x_1\ne x_2$, $y_1\ne y_2$, and $x_1-x_2=\lambda(y_1-y_2)$ for some $\lambda\in F^\times$, that is, if the two differences are nonzero and parallel; $\Tcal(A,B)$ denotes the number of trapezoids. Configurations with $x_1=x_2$ or $y_1=y_2$ are not counted.

With this algebraic definition, the rectangle estimates below require the quadratic form $u\mapsto u\cdot u$ on $F^2$ to be anisotropic, equivalently, $-1$ must not be a square in $F$. This hypothesis is necessary: if $v\ne0$ satisfies $v\cdot v=0$ and $A$ lies on the isotropic line $Fv$, then every ordered quadruple in $A^4$ is a rectangle, so $\Rcal(A)=|A|^4$. The trapezoid estimates do not require anisotropy. The following rectangle bounds follow from the standard incidence arguments and are stated only in the anisotropic setting in which they will be used.

\begin{lemma}\label{lem:rect}
Suppose that $-1$ is not a square in $F$. Then, for every $A\subset F^2$,
\begin{equation}\label{eq:rect-CS}
\Rcal(A)\lesssim|A|^{5/2},
\end{equation}
and
\begin{equation}\label{eq:rect-vinh}
\Rcal(A)\lesssim q^{-1}|A|^{3}+q^{1/2}|A|^{2},
\end{equation}
the second bound being the smaller for $q\le|A|\le q^2$.
\end{lemma}

The bound \eqref{eq:rect-CS} is an elementary Cauchy--Schwarz argument implicit in \cite{MT}, and \eqref{eq:rect-vinh} is an incidence estimate of Vinh \cite{Vinh} (see (15) of \cite{RS}).

\begin{lemma}[{\cite[Theorem 4]{LewRect}}]\label{lem:rect-prime}
Suppose that $F=\F_p$, that $-1$ is not a square in $F$, and that $|A|\le p^{26/21}$. Then, for every $\eps>0$
\begin{equation}\label{eq:rect-prime}
\Rcal(A)\lesssim_\eps|A|^{99/41+\eps}.
\end{equation}
\end{lemma}

The rectangles in \cite{LewRect} are required to be nondegenerate. Under the anisotropy hypothesis, any additional degenerate rectangles permitted by our algebraic definition contribute $O(|A|^2)$: if one edge difference vanishes, the opposite edge difference is isotropic and therefore also vanishes. This contribution is absorbed by \eqref{eq:rect-prime}.

\noindent\emph{Note added.} The sharp Szemer\'edi--Trotter theorem subsequently obtained in \cite{LewST} gives the stronger bound $\Rcal(A)\lesssim |A|^2\log(2|A|)+p^{-1}|A|^3$ for $A\subset\F_p^2$ with $-1$ nonsquare. The argument below uses only the earlier rectangle estimates recorded above.

It is convenient to organize a planar set according to its concentration on lines. Following \cite{LewBilinear}, a set $A\subset F^2$ is $k$-regular, for $k\le|A|^{1/2}$, if it is a disjoint union of $\sim k$ subsets, each lying on one of $\sim k$ distinct lines and containing $\sim|A|k^{-1}$ points, such that every line outside this frame of $\sim k$ lines meets $A$ in $O(k)$ points. All constants in this definition are fixed absolute constants. We call $A$ irregular if every line meets $A$ in at most $|A|^{1/2}$ points. The following records the estimates of Lemma 5.2 and Lemma 5.3 from \cite{LewBilinear}.

\begin{lemma}[{\cite{LewBilinear}}]\label{lem:regular}
Every $A\subset F^2$ is a disjoint union of $O(\log q)$ regular sets, with possibly repeated regularity parameters, together with at most one irregular set. If $A$ and $B$ are $k$-regular with $|A|\sim|B|\sim m$, or both irregular with $k\sim m^{1/2}$, then
\begin{equation}\label{eq:reg-T}
\Tcal(A,B)\lesssim m^3k+m^4k^{-2}.
\end{equation}
If, in addition, $-1$ is not a square in $F$, then
\begin{equation}\label{eq:reg-R}
\Rcal(A)\lesssim m^2k.
\end{equation}
\end{lemma}

For clarity, the decomposition can be obtained by repeatedly removing the points on a richest line whenever that line contains more than the square root of the number of remaining points. The remainder is irregular. Group the removed subsets by their dyadic cardinality $j\le |E|<2j$. At the first removal in such a group, fewer than $4j^2$ points remain, so the group contains fewer than $4j$ subsets. Split it into at most four collections of at most $j$ subsets each. A collection of $k$ subsets has cardinality $m\sim kj$, with $k\le j$ and hence $k\le m^{1/2}$. Its supporting lines are distinct, and any other line meets their union in at most $k$ points. Thus each collection is $k$-regular. There are $O(\log q)$ possible values of $j$, as required. Replacing $k$ by the largest dyadic integer not exceeding it only changes the absolute constants in the definition.

We will always carry out the slice regularization outside the bilinear estimate. More precisely, let $g\sim1_G$ on $F^3$, and decompose each horizontal support $G_z$ by Lemma \ref{lem:regular}. Number the components within each slice, and group them by this component index, a dyadic component-size range, and their dyadic regularity parameter, keeping the irregular components as a separate regularity class. The component index ensures that each group contains at most one component from a given slice. This gives a disjoint decomposition
\begin{equation}\label{eq:outside-regularization}
g=\sum_{\nu=1}^{M}g_\nu,\qquad M=O((\log q)^3),
\end{equation}
with the following property: for each fixed $\nu$, all nonempty horizontal slices of $\supp(g_\nu)$ have comparable cardinality and are either all $k$-regular for one dyadic $k$, or all irregular. We then use the triangle inequality
\begin{equation}\label{eq:outside-triangle}
\|\widehat g\|_{L^2(P,d\sigma)}
\le \sum_{\nu=1}^{M}\|\widehat {g_\nu}\|_{L^2(P,d\sigma)}.
\end{equation}
Thus the bilinear inequality is applied only within a single piece $g_\nu$; no cross-regularity terms occur. If a uniform restricted estimate $\|\widehat {g_\nu}\|_2\lesssim q^\eps|\supp(g_\nu)|^\alpha$ is proved for every piece, then
\begin{equation}\label{eq:outside-sum}
\sum_{\nu=1}^{M}|\supp(g_\nu)|^\alpha
\le M|G|^\alpha,
\end{equation}
and the factor $M$ is absorbed into an arbitrarily small power of $q$. This is the only regularization bookkeeping used below.

We next turn to the main bilinear tool from \cite{LewBilinear}, which should be thought of as a more delicate form of Lemma \ref{lem:slice}. Let $P=P_3$, let $g\sim1$ on $G\subset F^3$, and for $z\in F$ let $g_z$ be the restriction of $g$ to the slice at height $z$, with support $G_z\subset F^2$. When convolving with $K$, we regard $g_z$ as a function on $F^3$ supported on the plane $x_3=z$.

\begin{proposition}[{\cite[(4.2)]{LewBilinear}}]\label{prop:bilinear-ineq}
With the notation above,
\begin{equation}\label{eq:BiMT}
\|\widehat g\|_{L^2(P,d\sigma)}\lesssim|G|^{1/2}+|G|^{3/8}\Big(\sum_{z}\|g_z*K\|_4^2+\sum_{z\ne z'}\|g_z*K\,g_{z'}*K\|_2\Big)^{1/4},
\end{equation}
the linear term satisfies
\begin{equation}\label{eq:linear}
\|g_z*K\|_4^4\lesssim q^{-1}\Rcal(G_z),
\end{equation}
and the bilinear term satisfies
\begin{equation}\label{eq:bilinear}
\|g_z*K\,g_{z'}*K\|_2^2\lesssim q^{-1}|G_z||G_{z'}|+q^{-2}\Tcal(G_z,G_{z'})\qquad(z\ne z').
\end{equation}
\end{proposition}

Set
\begin{equation}\label{eq:B-def}
\Bcal(A,B)=\min\big\{\Tcal(A,B),\,q(\Rcal(A)\Rcal(B))^{1/2}\big\}.
\end{equation}
Combining \eqref{eq:bilinear} with the Cauchy--Schwarz consequence of \eqref{eq:linear} gives
\[
\|g_z*K\,g_{z'}*K\|_2^2
\lesssim q^{-1}|G_z||G_{z'}|+q^{-2}\Bcal(G_z,G_{z'}).
\]
Indeed, this is the elementary inequality $\min\{U+V,W\}\le U+\min\{V,W\}$, with $U=q^{-1}|G_z||G_{z'}|$, $V=q^{-2}\Tcal(G_z,G_{z'})$, and $W=q^{-1}(\Rcal(G_z)\Rcal(G_{z'}))^{1/2}$. Inserting the linear and bilinear bounds into \eqref{eq:BiMT}, for regular $g\sim1_G$ with $w$ nonempty slices each of size $m\sim|G|w^{-1}$, and using $\sum_{z\ne z'}|G_z|^{1/2}|G_{z'}|^{1/2}\le w|G|\le q|G|$, gives
\begin{equation}\label{eq:biMT-recall}
\|\widehat g\|_{L^2(P,d\sigma)}\lesssim|G|^{1/2}+|G|^{3/8}q^{-1/8}\Big(\sum_{z}\Rcal(G_z)^{1/2}+q^{-1/2}\sum_{z\ne z'}\Bcal(G_z,G_{z'})^{1/2}+q|G|\Big)^{1/4}.
\end{equation}
Since $\Tcal$ counts only nondegenerate parallel configurations, the degenerate ones, numbering $\lesssim|G_z||G_{z'}|$, are already accounted for by the term $q^{-1}|G_z||G_{z'}|$ in \eqref{eq:bilinear}. This term gives the residual contribution $q|G|$ in \eqref{eq:biMT-recall}; the diagonal slices $z=z'$ are accounted for by $\sum_z\Rcal(G_z)^{1/2}$.

The quantity \eqref{eq:B-def} is effective because a pair of sets cannot have both many trapezoids and many rectangles. This is made quantitative by the following estimate of \cite{LewBilinear}, which is the source of the arbitrary-field exponent $\frac{32}{9}$.

\begin{proposition}[{\cite[Prop. 5.4]{LewBilinear}}]\label{prop:general-B}
Suppose that $-1$ is not a square in $F$. Let $A,B\subset F^2$ be $k$-regular sets, or both irregular sets, with $|A|\sim|B|\sim m$. Then
\begin{equation}\label{eq:general-B}
\Bcal(A,B)\lesssim m^{8/3}q^{2/3}+m^{7/2}.
\end{equation}
\end{proposition}

Inserting Proposition \ref{prop:general-B} into \eqref{eq:biMT-recall} and combining with the standard estimates yields, for regular $g\sim1_G$,
\begin{equation}\label{eq:corrected-master}
\|\widehat g\|_{L^2(P,d\sigma)}\lesssim|G|^{1/2}+|G|^{11/16}q^{-1/8}+q^{-3/16}|G|^{13/16}+|G|^{17/24}+q^{1/8}|G|^{5/8},
\end{equation}
and from this the table
\begin{equation}\label{eq:corrected-table}
\|\widehat g\|_{L^2(P,d\sigma)}\lesssim|G|^{1/2}+\left\{\begin{array}{ll}
|G|q^{-1/2},& |G|\le q^{12/7},\\[2mm]
|G|^{17/24},& q^{12/7}\le|G|\le q^{9/5},\\[2mm]
q^{-3/16}|G|^{13/16},& q^{9/5}\le|G|\le q^2,\\[2mm]
q^{3/16}|G|^{5/8},& q^2\le|G|\le q^{5/2},\\[2mm]
q^{1/2}|G|^{1/2},& q^{5/2}\le|G|\le q^3.
\end{array}\right.
\end{equation}
The table is applied to each piece in the outside decomposition \eqref{eq:outside-regularization}. Every term of \eqref{eq:corrected-table} is bounded by $|G|^{23/32}$; hence \eqref{eq:outside-triangle}--\eqref{eq:outside-sum} give the same power bound for the original function, with only a $q^\eps$ loss. It follows that $R^*(2\to r)\lesssim_r1$ for $r>\frac{32}{9}$ in arbitrary fields of odd characteristic in which $-1$ is not a square.

In prime fields the rectangle bound \eqref{eq:rect-prime} improves the input still further, and the argument of \cite{LewRect} gives
\begin{equation}\label{eq:LewRect-prime-table}
\|\widehat g\|_{L^2(P,d\sigma)}\lesssim_\eps|G|^{1/2}+\left\{\begin{array}{ll}
|G|p^{-1/2},& |G|\le p^{94/53},\\[2mm]
p^{3/41}|G|^{111/164+\eps},& p^{94/53}\le|G|\le p^{75/34},\\[2mm]
p^{3/16}|G|^{5/8},& p^{75/34}\le|G|\le p^{5/2},\\[2mm]
p^{1/2}|G|^{1/2},& p^{5/2}\le|G|\le p^3.
\end{array}\right.
\end{equation}
This gives $R^*(2\to r)\lesssim_r1$ for every $r>\frac{188}{53}$, the prime-field estimate of \cite{LewRect}.

\section{The prime-field improvement}\label{sec:prime}

In this section $F=\F_p$ and $-1$ is not a square. The general field argument of Section \ref{sec:background} identifies the obstruction in \eqref{eq:corrected-master} as the term $q^{-3/16}|G|^{13/16}$, which arises from the $m^{7/2}$ summand of Proposition \ref{prop:general-B}. We improve this in prime fields by using a direction-sensitive point-line incidence theorem of Lund, Pham and Vinh \cite{LPV}.

Two nonzero vectors of $F^2$ have the same direction if one is a scalar multiple of the other, and a line has the direction of any nonzero vector along it. There are $p+1$ directions in all. For $A\subset F^2$ and a direction $\theta$ we set
\begin{equation}\label{eq:Etheta-def}
E_\theta(A)=|\{(a,a')\in A^2:a-a'\text{ has direction }\theta\}|,
\end{equation}
which counts ordered pairs of distinct points, since $0$ has no direction. With the trapezoid count of Section \ref{sec:background} restricted to nondegenerate configurations, this gives the identity $\Tcal(A,B)=\sum_\theta E_\theta(A)E_\theta(B)$. More generally, for a family $\mathcal L$ of lines we write $E_{\theta,\mathcal L}(A)$ for the number of ordered pairs $(a,a')\in A^2$ that lie on a common line of $\mathcal L$ of direction $\theta$. We refer to $\sum_\theta E_{\theta,\mathcal L}(A)^2$, and to $\sum_\theta E_\theta(A)^2$ when $\mathcal L$ is the family of all lines, as the directional energy of $A$. By Cauchy--Schwarz in $\theta$ the trapezoid count satisfies $\Tcal(A,B)\le\big(\sum_\theta E_\theta(A)^2\big)^{1/2}\big(\sum_\theta E_\theta(B)^2\big)^{1/2}$, so a bound on the directional energy can be used to bound the trapezoid count.

For a finite set $A\subset\F_p^2$ and a finite family $\mathcal L$ of lines, write
\[
I(A,\mathcal L)=\big|\{(a,\ell)\in A\times\mathcal L:a\in\ell\}\big|
\]
for the number of incidences, and let $t$ denote the number of distinct directions spanned by the lines of $\mathcal L$. We use two point-line incidence bounds, one of Lund, Pham and Vinh and one of Vinh.

\begin{lemma}[\cite{LPV}]\label{lem:LPV}
Let $A\subset\F_p^2$ with $|A|=m$ and let $\mathcal L$ be a family of $L$ lines spanning $t$ directions, with $t\le c_0p^2/m$ for a sufficiently small absolute constant $c_0>0$. Then
\begin{equation}\label{eq:LPV}
I(A,\mathcal L)\lesssim t^{2/9}L^{5/9}m^{7/9}+t^{1/2}L^{1/2}m^{1/2}+L^{2/3}m^{2/3}+L+m.
\end{equation}
\end{lemma}

The direction-sensitive incidence theorem of Lund, Pham and Vinh \cite[Theorem~1.5]{LPV} gives the same estimate with
\[
\min\{t^{3/10}L^{3/5}m^{7/10},\,t^{2/9}L^{5/9}m^{7/9}\}
\]
in place of the first term on the right side of \eqref{eq:LPV}. Since we only need an upper bound, we replace this minimum by its second entry. Next we recall Vinh's well-known incidence estimate.

\begin{lemma}[\cite{Vinh}]\label{lem:vinh}
For every $A\subset\F_p^2$ with $|A|=m$ and every family $\mathcal L$ of $L$ lines,
\begin{equation}\label{eq:vinh-inc}
I(A,\mathcal L)\lesssim \frac{mL}{p}+(pmL)^{1/2}.
\end{equation}
\end{lemma}

We will apply \eqref{eq:vinh-inc} to line families spanning $t>c_0p^2/m$ directions, outside the stated range of \eqref{eq:LPV}. We now put this together to obtain an estimate on the directional energy.

\begin{lemma}\label{lem:directional-energy}
Let $A\subset\F_p^2$ with $|A|=m$, and let $\mathcal L$ be a family of lines each meeting $A$ in at most $C m^{1/2}$ points, where $C$ is a fixed absolute constant. Then for every $\eps>0$,
\begin{equation}\label{eq:dir-energy}
\sum_\theta E_{\theta,\mathcal L}(A)^2\lesssim_\eps p^\eps\Big(m^3\min(m,p)^{1/3}+\frac{m^4}{p}+pm^2\Big).
\end{equation}
\end{lemma}

\begin{proof}
We may enlarge $\mathcal L$ to the family of all lines meeting $A$ in at most $C m^{1/2}$ points, since every ordered pair of distinct points determines a unique line and all terms are nonnegative. Constants in this proof may depend on $C$. Put $\mu=\min(m,p)$ and $\tau=\mu^{1/3}$. Split this family as
\[
\mathcal L=\mathcal L^-\sqcup\mathcal L^+,
\]
where the lines in $\mathcal L^-$ meet $A$ in at most $\tau$ points and those in $\mathcal L^+$ meet $A$ in more than $\tau$ points. Since
\[
E_{\theta,\mathcal L}(A)
=E_{\theta,\mathcal L^-}(A)+E_{\theta,\mathcal L^+}(A),
\]
the elementary inequality $(u+v)^2\le2u^2+2v^2$ gives
\begin{equation}\label{eq:directional-low-high-split}
\sum_\theta E_{\theta,\mathcal L}(A)^2
\le 2\sum_\theta E_{\theta,\mathcal L^-}(A)^2
 +2\sum_\theta E_{\theta,\mathcal L^+}(A)^2.
\end{equation}

We first bound the low-richness part. In a fixed direction, the lines of $\mathcal L^-$ are parallel and disjoint. If such a line contains $r\le\tau$ points of $A$, then it contributes $r(r-1)\le\tau r$ ordered pairs. Consequently,
\[
E_{\theta,\mathcal L^-}(A)\le\tau m.
\]
Moreover, every ordered pair of distinct points has a unique direction, so
$\sum_\theta E_{\theta,\mathcal L^-}(A)\le m^2$. Therefore
\begin{equation}\label{eq:low-rich}
\sum_\theta E_{\theta,\mathcal L^-}(A)^2
\le \tau m\sum_\theta E_{\theta,\mathcal L^-}(A)
\le \tau m^3=m^3\mu^{1/3}.
\end{equation}

It remains to estimate $\mathcal L^+$. For dyadic parameters $j$ and $s$, let $\mathcal C_{j,s}$ be the lines in $\mathcal L^+$ meeting $A$ in between $j$ and $2j$ points and whose direction contains between $s$ and $2s$ such lines. The nonempty classes form a collection $\mathscr C$ of cardinality $J=O((1+\log m)^2)$. For each direction,
\[
E_{\theta,\mathcal L^+}(A)=\sum_{\mathcal C\in\mathscr C}E_{\theta,\mathcal C}(A),
\]
and Cauchy--Schwarz over the classes gives
\begin{equation}\label{eq:class-cross-terms}
\sum_\theta E_{\theta,\mathcal L^+}(A)^2
\le J\sum_{\mathcal C\in\mathscr C}
       \sum_\theta E_{\theta,\mathcal C}(A)^2.
\end{equation}
Thus it is enough to prove the asserted bound uniformly for one class. More explicitly, if every class contributes at most $H$, then the right side of \eqref{eq:class-cross-terms} is at most $J^2H$. Since $J=O((1+\log m)^2)$, this costs at most $O((1+\log m)^4)$, which is absorbed into $p^\eps$ because $m\le p^2$.

Fix a nonempty class $\mathcal C=\mathcal C_{j,s}$, and let $t$ be the number of directions it spans. Then $j\gtrsim\tau$, the class contains $\sim s$ lines in each of these directions and $\sim ts$ lines in all, and each line contains $\sim j$ points of $A$. For a fixed direction the class contributes $O(sj^2)$ ordered pairs, and hence
\begin{equation}\label{eq:class-energy}
\sum_\theta E_{\theta,\mathcal C}(A)^2\lesssim t(sj^2)^2=ts^2j^4.
\end{equation}
We will use three constraints repeatedly. There are $p+1$ directions, so $t\lesssim p$; the lines of one direction are parallel and disjoint, so $sj\lesssim m$; and the hypothesis gives $j\lesssim m^{1/2}$.

In this class each of the $\sim ts$ lines carries at least $j$ points, so
\begin{equation}\label{eq:incidence-lower}
I(A,\mathcal C)\gtrsim jts.
\end{equation}
Whichever incidence bound applies, at least one term on its right side is $\gtrsim jts$. Suppose first that the class spans few directions, $t\le c_0p^2/m$, so that \eqref{eq:LPV} applies. If the term $t^{2/9}(ts)^{5/9}m^{7/9}$ in \eqref{eq:LPV} is $\gtrsim jts$, then
\begin{equation}\label{eq:LPV-class}
jts\lesssim t^{2/9}(ts)^{5/9}m^{7/9}\ \Longrightarrow\ t\lesssim m^{7/2}j^{-9/2}s^{-2}\ \Longrightarrow\ ts^2j^4\lesssim m^{7/2}j^{-1/2}.
\end{equation}
At the same time, Vinh's estimate \eqref{eq:vinh-inc} applies to this same class. If the term $mL/p\sim mts/p$ in \eqref{eq:vinh-inc} is $\gtrsim jts$, then $j\lesssim m/p$, and using $t\lesssim p$ and $sj\lesssim m$ gives
\[
ts^2j^4\lesssim p(m/j)^2j^4=pm^2j^2\lesssim \frac{m^4}{p}.
\]
If the term $(pmL)^{1/2}\sim(pmts)^{1/2}$ in \eqref{eq:vinh-inc} is $\gtrsim jts$, then $ts\lesssim pm/j^2$, and hence, using $sj\lesssim m$,
\[
ts^2j^4\lesssim (pmj^{-2})(mj^{-1})j^4=pm^2j.
\]
Consequently, in the case where the term $t^{2/9}(ts)^{5/9}m^{7/9}$ in \eqref{eq:LPV} is $\gtrsim jts$, we have
\begin{equation}\label{eq:LPV-vinh-combined}
ts^2j^4\lesssim \frac{m^4}{p}+\min\{m^{7/2}j^{-1/2},\,pm^2j\}.
\end{equation}
If $m\le p$, then $j\gtrsim\mu^{1/3}=m^{1/3}$ and the minimum in \eqref{eq:LPV-vinh-combined} is $\lesssim m^{10/3}=m^3\mu^{1/3}$. If $m\ge p$, put $j_0=mp^{-2/3}$. Since $j_0\ge p^{1/3}=\mu^{1/3}$, either $j\le j_0$, in which case $pm^2j\le m^3p^{1/3}$, or $j\ge j_0$, in which case $m^{7/2}j^{-1/2}\le m^3p^{1/3}$. Thus this case contributes at most a constant multiple of
\[
\frac{m^4}{p}+m^3\mu^{1/3}.
\]

We next handle the other terms in \eqref{eq:LPV}. If the term $t^{1/2}(ts)^{1/2}m^{1/2}$ is $\gtrsim jts$, then $s\lesssim m/j^2$, so $ts^2j^4\lesssim tm^2\lesssim pm^2$. If the term $(ts)^{2/3}m^{2/3}$ is $\gtrsim jts$, then $ts\lesssim m^2j^{-3}$, so $ts^2j^4\lesssim m^2(sj)\lesssim m^3\le m^3\mu^{1/3}$. If the term $m$ is $\gtrsim jts$, then $ts^2j^4=(jts)(sj^3)\lesssim m^3\le m^3\mu^{1/3}$, using $sj\lesssim m$ and $j\lesssim m^{1/2}$. Finally, the $L$ term in \eqref{eq:LPV} satisfies $L\sim ts$ for the present class. If this term is $\gtrsim jts$, then $j\lesssim1$; using $sj\lesssim m$ and $t\lesssim p$ gives $ts^2j^4\lesssim ts^2\lesssim t(m/j)^2\lesssim pm^2$. Thus each class with few directions satisfies
\begin{equation}\label{eq:few-dir}
ts^2j^4\lesssim m^3\mu^{1/3}+\frac{m^4}{p}+pm^2.
\end{equation}

Suppose instead that the class spans many directions, $t>c_0p^2/m$, which is outside the stated range of \eqref{eq:LPV}. Here \eqref{eq:vinh-inc} applies with $L\sim ts$, and by \eqref{eq:incidence-lower} one of its two terms is $\gtrsim jts$. If the term $mL/p\sim mts/p$ is $\gtrsim jts$, then $j\lesssim m/p$, so using $t\lesssim p$ and $sj\lesssim m$,
\begin{equation}\label{eq:many-dir}
ts^2j^4\lesssim p(m/j)^2j^4=pm^2j^2\lesssim \frac{m^4}{p}.
\end{equation}
If the term $(pmL)^{1/2}\sim(pmts)^{1/2}$ is $\gtrsim jts$, then $ts\lesssim pm/j^2$, hence $s\lesssim m^2/(pj^2)$ by $t>c_0p^2/m$, and again $ts^2j^4\lesssim p\big(m^2/(pj^2)\big)^2j^4=m^4/p$.

The estimate \eqref{eq:few-dir} or \eqref{eq:many-dir} bounds \eqref{eq:class-energy} uniformly for every class. Summing over the classes in \eqref{eq:class-cross-terms}, combining with \eqref{eq:low-rich} and \eqref{eq:directional-low-high-split}, and absorbing the resulting $O((1+\log m)^4)$ factor into $p^\eps$ gives \eqref{eq:dir-energy}.
\end{proof}

We now replace Proposition \ref{prop:general-B} in prime fields.

\begin{proposition}\label{prop:prime-B}
Suppose that $-1$ is not a square in $\F_p$. Let $A,B\subset\F_p^2$ be $k$-regular sets, or both irregular sets, with $|A|\sim|B|\sim m$. Then, for every $\eps>0$,
\begin{equation}\label{eq:prime-B}
\Bcal(A,B)\lesssim_\eps p^\eps\Big(m^{8/3}p^{2/3}+m^3\min(m,p)^{1/3}+\frac{m^4}{p}+pm^2\Big).
\end{equation}
\end{proposition}

\begin{proof}
Set $k_*=m^{2/3}p^{-1/3}$ and $\mu=\min(m,p)$.

If $A$ and $B$ are $k$-regular with $k\le k_*$, then \eqref{eq:reg-R} gives $\Rcal(A),\Rcal(B)\lesssim m^2k$, so the second term of \eqref{eq:B-def} already gives the result,
\[
\Bcal(A,B)\le p\big(\Rcal(A)\Rcal(B)\big)^{1/2}\lesssim pm^2k\le pm^2k_*=m^{8/3}p^{2/3}.
\]

Suppose instead that $A$ and $B$ are $k$-regular with $k\ge k_*$, and use the first term of \eqref{eq:B-def}, $\Bcal(A,B)\le\Tcal(A,B)=\sum_\theta E_\theta(A)E_\theta(B)$. Call a pair of distinct points a frame pair if the two points lie on a common frame line and a non-frame pair otherwise, and let $F_\theta(A)$ and $E_{\theta,\mathcal L_A}(A)$ count, respectively, the frame and non-frame pairs of $A$ of direction $\theta$, where $\mathcal L_A$ is the family of non-frame lines of $A$. Then $E_\theta(A)=F_\theta(A)+E_{\theta,\mathcal L_A}(A)$, and likewise for $B$, so expanding the product splits the trapezoid count into three parts,
\begin{equation}\label{eq:three-parts}
\Tcal(A,B)=\sum_\theta F_\theta(A)F_\theta(B)+\sum_\theta E_{\theta,\mathcal L_A}(A)E_{\theta,\mathcal L_B}(B)+\sum_\theta\big(F_\theta(A)E_{\theta,\mathcal L_B}(B)+E_{\theta,\mathcal L_A}(A)F_\theta(B)\big),
\end{equation}
the frame, non-frame, and mixed contributions, which we bound in turn. The non-frame lines $\mathcal L_A$ and $\mathcal L_B$ meet their sets in $O(k)=O(m^{1/2})$ points.

We record carefully the number of frame pairs. Write the regular decomposition of $A$ as
\[
A=\bigsqcup_{i=1}^{r}A_i,\qquad A_i\subset\ell_i,
\qquad r\sim k,\qquad |A_i|\sim m/k,
\]
where the frame lines $\ell_i$ are distinct. For a fixed $i$, each set $A_j$ with $j\ne i$ contributes at most one point to $\ell_i$, because two distinct lines meet in at most one point. Hence
\[
|A\cap\ell_i|\le |A_i|+O(k)
\lesssim \frac{m}{k}+k
\lesssim \frac{m}{k},
\]
where the last inequality uses $k\le m^{1/2}$. A distinct pair lies on at most one frame line, and therefore
\[
\sum_\theta F_\theta(A)
\le \sum_{i=1}^{r}|A\cap\ell_i|^2
\lesssim k\left(\frac{m}{k}\right)^2
=\frac{m^2}{k}.
\]
In particular,
\begin{equation}\label{eq:frame-energy}
\Big(\sum_\theta F_\theta(A)^2\Big)^{1/2}
\le\sum_\theta F_\theta(A)
\lesssim\frac{m^2}{k},
\end{equation}
and likewise for $B$.

For the frame contribution, Cauchy--Schwarz in $\theta$ and \eqref{eq:frame-energy} give
\[
\sum_\theta F_\theta(A)F_\theta(B)\le\Big(\sum_\theta F_\theta(A)^2\Big)^{1/2}\Big(\sum_\theta F_\theta(B)^2\Big)^{1/2}\lesssim\frac{m^4}{k^2}\le\frac{m^4}{k_*^2}=m^{8/3}p^{2/3}.
\]
For the non-frame contribution, Lemma \ref{lem:directional-energy} applied to $(A,\mathcal L_A)$ and to $(B,\mathcal L_B)$ and Cauchy--Schwarz in $\theta$ give
\[
\sum_\theta E_{\theta,\mathcal L_A}(A)E_{\theta,\mathcal L_B}(B)\le\Big(\sum_\theta E_{\theta,\mathcal L_A}(A)^2\Big)^{1/2}\Big(\sum_\theta E_{\theta,\mathcal L_B}(B)^2\Big)^{1/2}\lesssim_\eps p^\eps\Big(m^3\mu^{1/3}+\frac{m^4}{p}+pm^2\Big).
\]
For the mixed contribution, combining \eqref{eq:frame-energy} with the bound of Lemma \ref{lem:directional-energy} for $(B,\mathcal L_B)$ in the form
\[
\Big(\sum_\theta E_{\theta,\mathcal L_B}(B)^2\Big)^{1/2}\lesssim_\eps p^\eps\Big(m^{3/2}\mu^{1/6}+m^2p^{-1/2}+p^{1/2}m\Big)
\]
through Cauchy--Schwarz gives
\begin{equation}\label{eq:mixed}
\sum_\theta F_\theta(A)E_{\theta,\mathcal L_B}(B)
\lesssim_\eps p^\eps\,\frac{m^2}{k}\Big(m^{3/2}\mu^{1/6}+m^2p^{-1/2}+p^{1/2}m\Big).
\end{equation}
Using $k\ge k_*=m^{2/3}p^{-1/3}$, the three summands in \eqref{eq:mixed} are at most
\begin{equation}\label{eq:mixed-bounds}
m^{17/6}p^{1/3}\mu^{1/6},\qquad m^{10/3}p^{-1/6},\qquad m^{7/3}p^{5/6},
\end{equation}
respectively, and each is dominated by the right side of \eqref{eq:prime-B}. For the first term, if $m\le p$ it is $m^3p^{1/3}\le m^{8/3}p^{2/3}$, while if $m\ge p$ it is at most $m^3p^{1/3}=m^3\mu^{1/3}$. For the second term, if $m\le p$ it is at most $m^{10/3}=m^3\mu^{1/3}$; if $p\le m\le p^{3/2}$ it is at most $m^3p^{1/3}=m^3\mu^{1/3}$; and if $m\ge p^{3/2}$ it is at most $m^4/p$. For the third term, $m^{7/3}p^{5/6}\le m^{8/3}p^{2/3}+pm^2$ according as $m\ge p^{1/2}$ or $m\le p^{1/2}$. The other mixed term, with the roles of $A$ and $B$ exchanged, is identical. Combining the three contributions gives \eqref{eq:prime-B} for regular sets.

Finally, suppose that $A$ and $B$ are irregular. Let $\mathcal L$ be the family of all lines in $\F_p^2$. Since every line meets $A$ and $B$ in at most $m^{1/2}$ points, Lemma \ref{lem:directional-energy} applies to both $(A,\mathcal L)$ and $(B,\mathcal L)$. Also, because $\mathcal L$ is the family of all lines, every nonzero pair of points of $A$ or $B$ is counted in exactly one line of $\mathcal L$. Thus
\[
\Tcal(A,B)
=
\sum_\theta E_\theta(A)E_\theta(B)
=
\sum_\theta E_{\theta,\mathcal L}(A)E_{\theta,\mathcal L}(B).
\]
By Cauchy--Schwarz in $\theta$ and Lemma \ref{lem:directional-energy},
\[
\begin{aligned}
\Tcal(A,B)
&\le
\Big(\sum_\theta E_{\theta,\mathcal L}(A)^2\Big)^{1/2}
\Big(\sum_\theta E_{\theta,\mathcal L}(B)^2\Big)^{1/2} \\
&\lesssim_\eps
p^\eps\Big(m^3\mu^{1/3}+\frac{m^4}{p}+pm^2\Big),
\end{aligned}
\]
where we used $|A|\sim|B|\sim m$ in the last line. Since $\Bcal(A,B)\le\Tcal(A,B)$, this proves the irregular case.
\end{proof}

The preceding estimate controls the trapezoid side of \eqref{eq:B-def}.  At the scale that will be critical below, one can also improve the rectangle side by retaining the dependence on $p$ in the rich-line estimates, rather than passing directly to the estimate \eqref{eq:rect-prime}.

\begin{lemma}[Hybrid rectangle estimate]\label{lem:hybrid-rectangle}
Suppose that $-1$ is not a square in $\F_p$, and let $A\subset\F_p^2$ have cardinality $n$ with
\[
n\lesssim p^{10/7}.
\]
Then, for every $\eps>0$,
\begin{equation}\label{eq:hybrid-rectangle}
\Rcal(A)\lesssim_\eps p^\eps n^2p^{3/7}.
\end{equation}
\end{lemma}

\begin{proof}
We first dispose of degenerate rectangles.  If one of the four edge differences is zero, say $x_0=x_1$, then the two remaining corner conditions involving the opposite edge give $(x_3-x_2)\cdot(x_3-x_2)=0$.  Since $-1$ is not a square, the form $u\cdot u$ has no nonzero zeros, and hence $x_2=x_3$.  Thus degenerate rectangles are determined by two points of $A$, up to cyclic relabeling, and contribute $O(n^2)$, which is absorbed by the right side of \eqref{eq:hybrid-rectangle}. We may therefore count only rectangles whose edge-lines are genuine lines.

For a dyadic integer $k\ge2$, let $\mathcal L_k$ denote the lines containing between $k$ and $2k$ points of $A$.  The Stevens--de Zeeuw incidence theorem gives
\begin{equation}\label{eq:rich-SdZ}
|\mathcal L_k|\lesssim n^{11/4}k^{-15/4}+nk^{-1}+n^{13/2}p^{-15/2}.
\end{equation}
On the other hand, Vinh's estimate gives, whenever $k\gtrsim n/p$,
\begin{equation}\label{eq:rich-vinh}
|\mathcal L_k|\lesssim pnk^{-2}.
\end{equation}

We use the standard dyadic rectangle count from the proof of Lemma \ref{lem:rect-prime}.  Assign to each rectangle one of its edge-lines having maximal intersection with $A$.  If $K\ge1$, the rectangles assigned to lines containing at most $K$ points contribute
\begin{equation}\label{eq:rect-low-rich}
\lesssim n^2K.
\end{equation}
For a dyadic class $k\ge K$, the contribution is at most
\begin{equation}\label{eq:rect-rich-classes}
\begin{cases}
k^3|\mathcal L_k|,&K\le k\le n^{1/2},\\
nk|\mathcal L_k|,&k\ge n^{1/2}.
\end{cases}
\end{equation}
Indeed, the first line follows by choosing an assigned edge, one endpoint, and the adjacent perpendicular edge.  For the second line, the perpendicular lines through the points of a fixed edge are parallel and disjoint, so their intersections with $A$ have total size at most $n$.

Take a dyadic $K\sim C_0p^{3/7}$, where $C_0$ is a sufficiently large absolute constant. Since $n\lesssim p^{10/7}$, we may choose $C_0$ so that $K\gtrsim n/p$. If $K\ge n$, then \eqref{eq:rect-low-rich} already proves the result, so assume $K<n$. First suppose that $n\ge p^{8/7}$.  Applying \eqref{eq:rich-vinh} in \eqref{eq:rect-rich-classes}, both richness ranges contribute at most
\[
pn^{3/2}\le n^2p^{3/7}
\]
per dyadic class.  Together with \eqref{eq:rect-low-rich}, and after summing the $O(\log p)$ dyadic classes, this proves the claim in this case.

It remains to consider $n\le p^{8/7}$.  For $K\le k\le n^{1/2}$, combine \eqref{eq:rich-SdZ} and \eqref{eq:rich-vinh}. We use the elementary inequality
\[
\min\{U+V,W\}\le \min\{U,W\}+V
\qquad (U,V,W\ge0)
\]
to separate the leading term in \eqref{eq:rich-SdZ} from its two error terms. The leading terms satisfy
\[
\begin{aligned}
\min\{n^{11/4}k^{-3/4},\,pnk\}
&\le
\big(n^{11/4}k^{-3/4}\big)^{4/7}(pnk)^{3/7}\\
&=n^2p^{3/7}.
\end{aligned}
\]
The other two terms from \eqref{eq:rich-SdZ}, after multiplication by $k^3$, are bounded by
\[
nk^2+n^{13/2}p^{-15/2}k^3
\le n^2+n^8p^{-15/2}
\lesssim n^2p^{3/7},
\]
where the last inequality uses $n\le p^{8/7}$.  For $k\ge n^{1/2}$, \eqref{eq:rich-SdZ} and the second line of \eqref{eq:rect-rich-classes} give
\[
\begin{aligned}
n^{15/4}k^{-11/4}+n^2+n^{15/2}p^{-15/2}k
&\le n^{19/8}+n^2+n^{17/2}p^{-15/2}\\
&\lesssim n^2p^{3/7},
\end{aligned}
\]
where we used $k\le n$ and $n\le p^{8/7}$.  Summing the $O(\log p)$ dyadic classes and absorbing the logarithm into $p^\eps$ proves \eqref{eq:hybrid-rectangle}.
\end{proof}

\begin{corollary}\label{cor:hybrid-B}
Suppose that $-1$ is not a square in $\F_p$. Let $A,B\subset\F_p^2$ be $k$-regular sets, or both irregular sets, with $|A|\sim|B|\sim m$ and
\[
m\lesssim p^{10/7}.
\]
Then, for every $\eps>0$,
\begin{equation}\label{eq:hybrid-B}
\Bcal(A,B)\lesssim_\eps p^\eps\left(
m^{8/3}p^{2/3}+\frac{m^4}{p}+pm^2+
\min\{m^3p^{1/3},\,m^2p^{10/7}\}\right).
\end{equation}
\end{corollary}

\begin{proof}
Proposition \ref{prop:prime-B}, together with $\min(m,p)^{1/3}\le p^{1/3}$, gives the right side of \eqref{eq:hybrid-B} with $m^3p^{1/3}$ in place of the minimum.  Lemma \ref{lem:hybrid-rectangle} and the second alternative in \eqref{eq:B-def} give independently
\[
\Bcal(A,B)\le p\big(\Rcal(A)\Rcal(B)\big)^{1/2}
\lesssim_\eps p^\eps m^2p^{10/7}.
\]
For nonnegative $U,V,D$ one has $\min\{U+V,D\}\le U+\min\{V,D\}$.  Applying this with
\[
U=m^{8/3}p^{2/3}+m^4p^{-1}+pm^2,
\qquad V=m^3p^{1/3},
\qquad D=m^2p^{10/7},
\]
proves the result.
\end{proof}

We insert Proposition \ref{prop:prime-B} and Corollary \ref{cor:hybrid-B} into the bilinear inequality. Apply the outside decomposition \eqref{eq:outside-regularization} and the triangle inequality \eqref{eq:outside-triangle}. It is enough to estimate a single piece, which we relabel as $g\sim1_G$. Write $N=|G|$. Suppose this piece has $w$ nonempty horizontal slices. Their supports all lie in one common regularity class and have a common dyadic size
\[
m\sim Nw^{-1}.
\]
Thus Proposition \ref{prop:prime-B} applies to every pair of nonempty slices, and no cross-regularity terms occur. In the exponent calculations below we write $m=Nw^{-1}$; replacing $m$ by a comparable quantity only changes the implicit constants.

We start from \eqref{eq:biMT-recall}, with $q=p$, namely
\begin{equation}\label{eq:prime-start}
\|\widehat g\|_{L^2(P,d\sigma)}
\lesssim
N^{1/2}+N^{3/8}p^{-1/8}
\Big(\sum_z\Rcal(G_z)^{1/2}
+p^{-1/2}\sum_{z\ne z'}\Bcal(G_z,G_{z'})^{1/2}
+pN\Big)^{1/4}.
\end{equation}
We estimate the three terms inside the fourth root separately, using $(U+V+W)^{1/4}\le U^{1/4}+V^{1/4}+W^{1/4}$. For the linear slice term, \eqref{eq:rect-CS} gives $\Rcal(G_z)\lesssim m^{5/2}$ for each nonempty slice, and hence
\begin{equation}\label{eq:prime-linear-slices}
\sum_z\Rcal(G_z)^{1/2}
\lesssim wm^{5/4}=w(Nw^{-1})^{5/4}=N^{5/4}w^{-1/4}\le N^{5/4}.
\end{equation}
The contribution of this term to \eqref{eq:prime-start} is therefore at most
\[
N^{3/8}p^{-1/8}(N^{5/4})^{1/4}=p^{-1/8}N^{11/16}.
\]
The last term inside the fourth root in \eqref{eq:prime-start} contributes
\[
N^{3/8}p^{-1/8}(pN)^{1/4}=p^{1/8}N^{5/8}.
\]

It remains to estimate the bilinear slice term.  The same calculation will be
used for each term in the bound for $\Bcal(G_z,G_{z'})$.  If one such term is
$X(m,p)$, then after taking the square root and summing over the at most $w^2$
ordered pairs of nonempty slices, it gives
\[
p^{-1/2}w^2X(m,p)^{1/2}
\]
inside the fourth root in \eqref{eq:prime-start}.  Thus its contribution to
\eqref{eq:prime-start} is
\begin{equation}\label{eq:bilinear-contribution-rule}
\begin{aligned}
N^{3/8}p^{-1/8}\Big(p^{-1/2}w^2X(m,p)^{1/2}\Big)^{1/4}
&=p^{-1/4}N^{3/8}w^{1/2}X(m,p)^{1/8}.
\end{aligned}
\end{equation}
Since $(X_1+\cdots+X_r)^{1/2}\le X_1^{1/2}+\cdots+X_r^{1/2}$ for nonnegative $X_i$, we may apply this calculation term by term to the bound in Proposition \ref{prop:prime-B}, relabeling the harmless $p^\eps$ loss.

First take $X=m^{8/3}p^{2/3}$.  Since $m=Nw^{-1}$ and $w\le p$, \eqref{eq:bilinear-contribution-rule} gives
\begin{equation}\label{eq:bilinear-first-summand}
\begin{aligned}
p^{-1/4}N^{3/8}w^{1/2}(m^{8/3}p^{2/3})^{1/8}
&=p^{-1/6}N^{3/8}w^{1/2}m^{1/3} \\
&=p^{-1/6}N^{17/24}w^{1/6} \\
&\le N^{17/24}.
\end{aligned}
\end{equation}
Next take $X=m^4p^{-1}$.  Then
\begin{equation}\label{eq:bilinear-second-summand}
\begin{aligned}
p^{-1/4}N^{3/8}w^{1/2}(m^4p^{-1})^{1/8}
&=p^{-3/8}N^{3/8}w^{1/2}m^{1/2} \\
&=p^{-3/8}N^{7/8}.
\end{aligned}
\end{equation}
Finally take $X=pm^2$.  This gives
\begin{equation}\label{eq:bilinear-third-summand}
\begin{aligned}
p^{-1/4}N^{3/8}w^{1/2}(pm^2)^{1/8}
&=p^{-1/8}N^{3/8}w^{1/2}m^{1/4} \\
&=p^{-1/8}N^{5/8}w^{1/4} \\
&\le p^{1/8}N^{5/8}.
\end{aligned}
\end{equation}

It remains to handle the summand $m^3\min(m,p)^{1/3}$ in Proposition \ref{prop:prime-B}, together with the improvement supplied by Corollary \ref{cor:hybrid-B}.  If $N\le p^2$, then $m^3\min(m,p)^{1/3}\le m^{10/3}$.  Using \eqref{eq:bilinear-contribution-rule}, its contribution is at most
\begin{equation}\label{eq:bad-term-small-N}
\begin{aligned}
p^{-1/4}N^{3/8}w^{1/2}(m^{10/3})^{1/8}
&=p^{-1/4}N^{3/8}w^{1/2}m^{5/12} \\
&=p^{-1/4}N^{19/24}w^{1/12} \\
&\le p^{-1/6}N^{19/24} \\
&\le N^{17/24},
\end{aligned}
\end{equation}
where the last two inequalities use $w\le p$ and $N\le p^2$.  We may therefore assume from now on that $N\ge p^2$.  Since $w\le p$, this implies $m=Nw^{-1}\ge p$ on every nonempty slice.

For slice sizes in the range of Corollary \ref{cor:hybrid-B}, the remaining part of the bilinear input is
\begin{equation}\label{eq:hybrid-input}
\min\{m^3p^{1/3},m^2p^{10/7}\}.
\end{equation}
The contribution obtained from the first term in \eqref{eq:hybrid-input} is
\begin{equation}\label{eq:hybrid-first-contribution}
\begin{aligned}
p^{-1/4}N^{3/8}w^{1/2}(m^3p^{1/3})^{1/8}
&=p^{-5/24}N^{3/8}w^{1/2}m^{3/8} \\
&=p^{-5/24}N^{3/4}w^{1/8}.
\end{aligned}
\end{equation}
The contribution obtained from the second term in \eqref{eq:hybrid-input} is
\begin{equation}\label{eq:hybrid-second-contribution}
\begin{aligned}
p^{-1/4}N^{3/8}w^{1/2}(m^2p^{10/7})^{1/8}
&=p^{-1/14}N^{3/8}w^{1/2}m^{1/4} \\
&=p^{-1/14}N^{5/8}w^{1/4}.
\end{aligned}
\end{equation}
The two inputs in \eqref{eq:hybrid-input} are equal when
\begin{equation}\label{eq:critical-slice-size}
m=p^{23/21}.
\end{equation}
Equivalently, since $m=Nw^{-1}$, the critical number of slices is
\begin{equation}\label{eq:critical-slice-number}
w_0=Np^{-23/21}.
\end{equation}
For $w\ge w_0$ the first term in \eqref{eq:hybrid-input} is smaller, while for $w\le w_0$ the second term is smaller.

We first dispose of the range $m\gtrsim p^{10/7}$, where Corollary \ref{cor:hybrid-B} is not being used.  This range corresponds to $w\lesssim Np^{-10/7}$.  Using the contribution \eqref{eq:hybrid-first-contribution}, we get
\begin{equation}\label{eq:large-slice-revert}
\begin{aligned}
p^{-5/24}N^{3/4}w^{1/8}
&\lesssim p^{-5/24}N^{3/4}(Np^{-10/7})^{1/8} \\
&=p^{-65/168}N^{7/8} \\
&\le p^{-3/8}N^{7/8}.
\end{aligned}
\end{equation}
This is already covered by \eqref{eq:bilinear-second-summand}.

It remains to consider the range $m\lesssim p^{10/7}$.  Suppose first that
\begin{equation}\label{eq:medium-N-range}
p^2\le N\le p^{44/21}.
\end{equation}
Then $w_0\le p$.  If $w\le w_0$, we use the second contribution \eqref{eq:hybrid-second-contribution}; by \eqref{eq:critical-slice-number},
\begin{equation}\label{eq:medium-range-second-input}
\begin{aligned}
p^{-1/14}N^{5/8}w^{1/4}
&\le p^{-1/14}N^{5/8}(Np^{-23/21})^{1/4} \\
&=p^{-29/84}N^{7/8} \\
&=p^{-1/12}N^{3/4}\big(p^{-11/42}N^{1/8}\big) \\
&\le p^{-1/12}N^{3/4},
\end{aligned}
\end{equation}
where the last inequality uses $N\le p^{44/21}$.  If instead $w\ge w_0$, we use the first contribution \eqref{eq:hybrid-first-contribution}; since $w\le p$,
\begin{equation}\label{eq:medium-range-first-input}
p^{-5/24}N^{3/4}w^{1/8}
\le p^{-5/24}N^{3/4}p^{1/8}
=p^{-1/12}N^{3/4}.
\end{equation}
Thus the hybrid part contributes at most
\begin{equation}\label{eq:hybrid-medium-bound}
p^{-1/12}N^{3/4}
\qquad (p^2\le N\le p^{44/21}).
\end{equation}

Now suppose that $N\ge p^{44/21}$.  Then $w_0\ge p$, so for all $w\le p$ in the range $m\lesssim p^{10/7}$, the second term in \eqref{eq:hybrid-input} is the smaller one.  Using \eqref{eq:hybrid-second-contribution} and $w\le p$ gives
\begin{equation}\label{eq:hybrid-large-bound}
p^{-1/14}N^{5/8}w^{1/4}
\le p^{-1/14}N^{5/8}p^{1/4}
=p^{5/28}N^{5/8}.
\end{equation}

Combining \eqref{eq:prime-linear-slices}, \eqref{eq:bilinear-first-summand}, \eqref{eq:bilinear-second-summand}, \eqref{eq:bilinear-third-summand}, \eqref{eq:bad-term-small-N}, \eqref{eq:large-slice-revert}, \eqref{eq:hybrid-medium-bound}, \eqref{eq:hybrid-large-bound}, and the residual bilinear term $pN$ in \eqref{eq:prime-start}, we obtain the following three estimates.  If $N\le p^2$, then
\begin{equation}\label{eq:prime-master-small-N}
\|\widehat g\|_{L^2(P,d\sigma)}
\lesssim_\eps p^\eps\Big(
N^{1/2}+p^{-1/8}N^{11/16}+N^{17/24}
+p^{1/8}N^{5/8}+p^{-3/8}N^{7/8}\Big).
\end{equation}
If $p^2\le N\le p^{44/21}$, then
\begin{equation}\label{eq:prime-master-medium-N}
\|\widehat g\|_{L^2(P,d\sigma)}
\lesssim_\eps p^\eps\Big(
N^{1/2}+p^{-1/8}N^{11/16}+N^{17/24}
+p^{1/8}N^{5/8}+p^{-3/8}N^{7/8}
+p^{-1/12}N^{3/4}\Big).
\end{equation}
If $N\ge p^{44/21}$, then
\begin{equation}\label{eq:prime-master-large-N}
\|\widehat g\|_{L^2(P,d\sigma)}
\lesssim_\eps p^\eps\Big(
N^{1/2}+p^{-1/8}N^{11/16}+N^{17/24}
+p^{1/8}N^{5/8}+p^{-3/8}N^{7/8}
+p^{5/28}N^{5/8}\Big).
\end{equation}

We combine \eqref{eq:prime-master-small-N}--\eqref{eq:prime-master-large-N} with the small-set estimate
\[
\|\widehat g\|_{L^2(P,d\sigma)}\lesssim N^{1/2}+Np^{-1/2}
\]
and with the prime-field estimate \eqref{eq:LewRect-prime-table} from \cite{LewRect}.  The part of \eqref{eq:LewRect-prime-table} that will be used below is the row
\[
p^{3/16}N^{5/8},\qquad p^{75/34}\le N\le p^{5/2}.
\]
The term $p^{-1/8}N^{11/16}$ is always bounded by $N^{17/24}$, and
\begin{equation}\label{eq:p18-term-dominated}
p^{1/8}N^{5/8}\le N^{17/24}
\qquad\text{whenever}\qquad
N\ge p^{3/2}.
\end{equation}
Also
\begin{equation}\label{eq:pminus38-transition}
p^{-3/8}N^{7/8}\le N^{17/24}
\quad\Longleftrightarrow\quad
N\le p^{9/4},
\end{equation}
and
\begin{equation}\label{eq:p528-transition}
p^{5/28}N^{5/8}\le N^{17/24}
\quad\Longleftrightarrow\quad
N\ge p^{15/7}.
\end{equation}
The middle hybrid term meets $N^{17/24}$ at $N=p^2$:
\begin{equation}\label{eq:middle-meets-pure}
p^{-1/12}N^{3/4}=N^{17/24}
\quad\Longleftrightarrow\quad
N=p^2.
\end{equation}
The two hybrid terms meet at
\begin{equation}\label{eq:new-transition}
p^{-1/12}N^{3/4}=p^{5/28}N^{5/8}
\quad\Longleftrightarrow\quad
N=p^{44/21}.
\end{equation}
For completeness, the remaining comparisons in the two hybrid ranges are also immediate from the endpoints. On $p^2\le N\le p^{44/21}$, the term $p^{-1/12}N^{3/4}$ dominates $N^{17/24}$, $p^{1/8}N^{5/8}$, and $p^{-3/8}N^{7/8}$, since respectively $N\ge p^2$, $N\ge p^{5/3}$, and $N\le p^{7/3}$. On $p^{44/21}\le N\le p^{15/7}$, the term $p^{5/28}N^{5/8}$ dominates $N^{17/24}$ and $p^{-3/8}N^{7/8}$, since $N\le p^{15/7}<p^{31/14}$. Finally, since $75/34<9/4$, the row $p^{3/16}N^{5/8}$ in \eqref{eq:LewRect-prime-table} is available for every $N$ with $p^{9/4}\le N\le p^{5/2}$, and it agrees with $N^{17/24}$ when $N=p^{9/4}$.  We therefore arrive at
\begin{equation}\label{eq:new-prime-table}
\|\widehat g\|_{L^2(P,d\sigma)}
\lesssim_\eps p^\eps\Bigg(N^{1/2}+
\left\{\begin{array}{ll}
Np^{-1/2},&N\le p^{12/7},\\[2mm]
N^{17/24},&p^{12/7}\le N\le p^2,\\[2mm]
p^{-1/12}N^{3/4},&p^2\le N\le p^{44/21},\\[2mm]
p^{5/28}N^{5/8},&p^{44/21}\le N\le p^{15/7},\\[2mm]
N^{17/24},&p^{15/7}\le N\le p^{9/4},\\[2mm]
p^{3/16}N^{5/8},&p^{9/4}\le N\le p^{5/2},\\[2mm]
p^{1/2}N^{1/2},&p^{5/2}\le N\le p^3.
\end{array}\right.\Bigg).
\end{equation}

For the moment, \eqref{eq:new-prime-table} is an estimate for a single piece in the outside decomposition. We next extract from it a uniform power bound, which can then be summed over all pieces.

Set
\[
\alpha=\frac{125}{176}.
\]
Every entry of \eqref{eq:new-prime-table} is bounded by $N^\alpha$ in its stated range.  For the first entry,
\begin{equation}\label{eq:first-alpha-check}
Np^{-1/2}\le N^\alpha
\quad\Longleftrightarrow\quad
N\le p^{88/51},
\end{equation}
and $12/7<88/51$.  The pure power $N^{17/24}$ is admissible because $17/24<125/176$.  For the two new entries,
\begin{equation}\label{eq:middle-alpha-check}
p^{-1/12}N^{3/4}\le N^\alpha
\quad\Longleftrightarrow\quad
N\le p^{44/21},
\end{equation}
while
\begin{equation}\label{eq:large-alpha-check}
p^{5/28}N^{5/8}\le N^\alpha
\quad\Longleftrightarrow\quad
N\ge p^{44/21}.
\end{equation}
For the term imported from \eqref{eq:LewRect-prime-table},
\begin{equation}\label{eq:LewRect-term-alpha-check}
p^{3/16}N^{5/8}\le N^\alpha
\quad\Longleftrightarrow\quad
N\ge p^{11/5},
\end{equation}
and $9/4>11/5$.  Lastly, $p^{1/2}N^{1/2}\le N^{7/10}<N^\alpha$ for $N\ge p^{5/2}$. Thus every reduced piece satisfies
\begin{equation}\label{eq:piece-uniform-power}
\|\widehat g\|_{L^2(P,d\sigma)}\lesssim_\eps p^\eps N^{125/176}.
\end{equation}

Return now to the original restricted function $g^{(0)}\sim1_{G^{(0)}}$ and write $g^{(0)}=\sum_{\nu=1}^{M}g_\nu$ as in \eqref{eq:outside-regularization}. Apply \eqref{eq:piece-uniform-power} with $\eps/2$ to each piece. Since $M=O((\log p)^3)\lesssim_\eps p^{\eps/2}$, the triangle inequality and \eqref{eq:outside-sum} give
\[
\begin{aligned}
\|\widehat {g^{(0)}}\|_{L^2(P,d\sigma)}
&\le \sum_{\nu=1}^{M}\|\widehat {g_\nu}\|_{L^2(P,d\sigma)}\\
&\lesssim_\eps p^{\eps/2}\sum_{\nu=1}^{M}|\supp(g_\nu)|^{125/176}\\
&\le p^{\eps/2}M|G^{(0)}|^{125/176}\\
&\lesssim_\eps p^\eps|G^{(0)}|^{125/176}.
\end{aligned}
\]
By Lemma \ref{lem:eps-removal}, \eqref{eq:dual} holds with $r'=s$ for every $1\le s<176/125$. Dualizing gives
\[
R^*(2\to r)\lesssim_r1
\qquad\text{for every}\qquad
r>\frac{176}{51},
\]
which proves Theorem \ref{thm:prime-main}.

\section{The endpoint estimate in six dimensions}\label{sec:six}

Let $P=P_6=\{(\xi,\xi\cdot \xi):\xi\in F^5\}\subset F^6$. We prove $R^*\big(2\to\frac83\big)\lesssim1$. The only nontrivial input is the additive-energy estimate of \cite[Lemma 2.2]{IKL}, namely
\begin{equation}\label{eq:six-energy}
\Lambda(E)\lesssim q^{-1}|E|^3+q^2|E|^2,\qquad E\subset P_6,
\end{equation}
which is sharp in general because $P_6$ contains affine two-planes. We write $\Qcal(f,h)=\langle\widehat f,\widehat h\rangle_{L^2(P,d\sigma)}=\langle f,h*(d\sigma)^\vee\rangle_{F^6}$ for the associated bilinear form. The argument below is bilinear, of the type used in \cite{LL} to reach endpoint estimates without a logarithmic loss.

\begin{proposition}\label{prop:six-bilinear}
Let $A,B\subset F^6$ and let $|f|\le1_A$, $|h|\le1_B$. Then
\begin{equation}\label{eq:six-bilinear-one-sided}
|\Qcal(f,h)|\lesssim|A\cap B|+q^{-1}|A|^{3/4}|B|^{3/4}+|A|^{3/4}|B|^{1/2},
\end{equation}
and consequently
\begin{equation}\label{eq:six-bilinear-symmetric}
|\Qcal(f,h)|\lesssim|A\cap B|+q^{-1}|A|^{3/4}|B|^{3/4}+\min\{|A|^{3/4}|B|^{1/2},|A|^{1/2}|B|^{3/4}\}.
\end{equation}
We also have the (cruder) estimate
\begin{equation}\label{eq:six-parseval}
|\Qcal(f,h)|\lesssim q|A|^{1/2}|B|^{1/2}.
\end{equation}
\end{proposition}

\begin{proof}
Since $(d\sigma)^\vee=\delta_0+K$, the contribution of $\delta_0$ to $\Qcal(f,h)=\langle f,h*(d\sigma)^\vee\rangle$ is $\langle f,h\rangle$, of modulus at most $|A\cap B|$. For the contribution of $K$, H\"older's inequality with exponents $\tfrac43$ and $4$ gives
\begin{equation}\label{eq:six-holder}
|\langle f,h*K\rangle|\le\|f\|_{4/3}\,\|h*K\|_4\le|A|^{3/4}\,\|h*K\|_4.
\end{equation}
To estimate $\|h*K\|_4$ we use the slice estimate of Lemma \ref{lem:slice} with $d=6$. For $z\in F$ let $h_z(\eta)=h(\eta,z)$ for $\eta\in F^5$, supported on $B_z=\{\eta\in F^5:(\eta,z)\in B\}$, and let $\widetilde h_z$ be its lift to the paraboloid, defined by $\widetilde h_z(\eta,\eta\cdot\eta)=h_z(\eta)$ on its support
\[
\widetilde B_z=\{(\eta,\eta\cdot\eta):\eta\in B_z\}\subset P,
\]
so that $|\widetilde B_z|=|B_z|$ and $\sum_z|B_z|=|B|$. Lemma \ref{lem:slice} gives
\begin{equation}\label{eq:six-slice}
\|h*K\|_4\lesssim q^{5/2}\sum_z\big\|(\widetilde h_zd\sigma)^\vee\big\|_4.
\end{equation}
For any function $\phi:P\to\C$ with $|\phi|\le1_E$ for some set $E\subset P$, expanding the fourth power and using Plancherel on $F^6$ identifies $\|(\phi d\sigma)^\vee\|_4^4$ with a normalized weighted count of additive quadruples in $E$, the quadruple $(\xi_1,\xi_2,\xi_3,\xi_4)$ with $\xi_1+\xi_2=\xi_3+\xi_4$ carrying the weight $\phi(\xi_1)\phi(\xi_2)\overline{\phi(\xi_3)\phi(\xi_4)}$. Since $|\phi|\le1_E$, each weight has modulus at most $1$, so this is at most the unweighted additive energy $\Lambda(E)$, whence
\begin{equation}\label{eq:six-L4}
\big\|(\phi d\sigma)^\vee\big\|_4\le q^{-7/2}\Lambda(E)^{1/4}.
\end{equation}
Applying the energy estimate \eqref{eq:six-energy} to $E=\widetilde B_z$ and $\phi=\widetilde h_z$ gives $\Lambda(\widetilde B_z)^{1/4}\lesssim q^{-1/4}|B_z|^{3/4}+q^{1/2}|B_z|^{1/2}$, so combining \eqref{eq:six-slice} and \eqref{eq:six-L4},
\begin{equation}\label{eq:six-hK}
\|h*K\|_4\lesssim q^{-1}\sum_z\Lambda(\widetilde B_z)^{1/4}\lesssim q^{-5/4}\sum_z|B_z|^{3/4}+q^{-1/2}\sum_z|B_z|^{1/2}.
\end{equation}
There are at most $q$ nonempty slices, so H\"older's inequality gives $\sum_z|B_z|^{3/4}\le q^{1/4}|B|^{3/4}$ and $\sum_z|B_z|^{1/2}\le q^{1/2}|B|^{1/2}$, and the right side of \eqref{eq:six-hK} reduces to
\begin{equation}\label{eq:six-hK2}
\|h*K\|_4\lesssim q^{-1}|B|^{3/4}+|B|^{1/2}.
\end{equation}
Substituting \eqref{eq:six-hK2} into \eqref{eq:six-holder} gives \eqref{eq:six-bilinear-one-sided}, and interchanging the roles of $f$ and $h$ and keeping the smaller of the two bounds gives \eqref{eq:six-bilinear-symmetric}. Finally, by Plancherel on $F^6$,
\[
\|\widehat f\|_{L^2(P,d\sigma)}^2=q^{-5}\sum_{\xi\in P}|\widehat f(\xi)|^2\le q^{-5}\sum_{\xi\in F^6}|\widehat f(\xi)|^2=q\|f\|_2^2,
\]
so the Cauchy--Schwarz inequality $|\Qcal(f,h)|\le\|\widehat f\|_{L^2(P,d\sigma)}\|\widehat h\|_{L^2(P,d\sigma)}$ gives \eqref{eq:six-parseval}.
\end{proof}

We now prove the endpoint. By duality, $R^*(2\to\frac83)\lesssim1$ is equivalent to $\|\widehat g\|_{L^2(P,d\sigma)}\lesssim\|g\|_{L^{8/5}(F^6)}$, the exponent $\frac85$ being conjugate to $\frac83$. Normalize $\sum_x|g(x)|^{8/5}=1$, so that $|g|\le1$, and decompose $g$ into dyadic level sets $G_i=\{x:2^{-i-1}<|g(x)|\le2^{-i}\}$ for $i\ge0$, writing $g_i=g1_{G_i}=2^{-i}h_i$ with $|h_i|\le1_{G_i}$. Set
\[
a_i=2^{-8i/5}|G_i|.
\]
Then $\sum_i a_i\lesssim1$, and in particular $a_i\lesssim1$ for every $i$. We will use the identity
\[
|G_i|=2^{8i/5}a_i
\]
each time a cardinality $|G_i|$ is rewritten below. Expanding the square gives
\begin{equation}\label{eq:six-double-sum}
\|\widehat g\|_{L^2(P,d\sigma)}^2=\sum_{i,j\ge0}2^{-i-j}\Qcal(h_i,h_j)\le\sum_{i,j\ge0}2^{-i-j}|\Qcal(h_i,h_j)|.
\end{equation}
Inserting the symmetric bound \eqref{eq:six-bilinear-symmetric}, with $A=G_i$ and $B=G_j$, into \eqref{eq:six-double-sum} gives the three terms
\begin{equation}\label{eq:six-three}
\begin{aligned}
\sum_{i,j\ge0}2^{-i-j}|G_i\cap G_j|
&+q^{-1}\sum_{i,j\ge0}2^{-i-j}|G_i|^{3/4}|G_j|^{3/4}\\
&+\sum_{i,j\ge0}2^{-i-j}\min\big\{|G_i|^{3/4}|G_j|^{1/2},|G_i|^{1/2}|G_j|^{3/4}\big\}.
\end{aligned}
\end{equation}
We will bound the three displayed terms separately. The first and third terms can be directly summed. The middle term is more subtle at the endpoint, and we will also use the Parseval estimate \eqref{eq:six-parseval} for each pair of dyadic levels.

The first term contributes only when $i=j$, since the level sets $G_i$ are disjoint. Thus
\begin{equation}\label{eq:six-diag}
\sum_{i,j\ge0}2^{-i-j}|G_i\cap G_j|=\sum_i2^{-2i}|G_i|=\sum_i2^{-2i}(2^{8i/5}a_i)=\sum_i2^{-2i/5}a_i\le\sum_i a_i\lesssim1.
\end{equation}

Next consider the last term in \eqref{eq:six-three}. Suppose first that $i\le j$, and use the first term of the minimum. This gives us:
\[
\begin{aligned}
2^{-i-j}|G_i|^{3/4}|G_j|^{1/2}
&=2^{-i-j}(2^{8i/5}a_i)^{3/4}(2^{8j/5}a_j)^{1/2}\\
&=2^{-(j-i)/5}a_i^{3/4}a_j^{1/2}\\
&\lesssim 2^{-(j-i)/5}a_i^{1/2}a_j^{1/2},
\end{aligned}
\]
where the last inequality uses $a_i\lesssim1$. The case $j\le i$ is identical, using the second term of the minimum. Therefore the total contribution of the last term in \eqref{eq:six-three} is at most
\[
\sum_{i,j\ge0}2^{-|i-j|/5}a_i^{1/2}a_j^{1/2}.
\]
Applying $a_i^{1/2}a_j^{1/2}\le\frac12(a_i+a_j)$ and using the symmetry in $i$ and $j$ gives
\begin{equation}\label{eq:six-asym}
\sum_{i,j\ge0}2^{-|i-j|/5}a_i^{1/2}a_j^{1/2}\le\sum_{i,j\ge0}2^{-|i-j|/5}a_i=\sum_{i\ge0}a_i\sum_{j\ge0}2^{-|i-j|/5}\lesssim\sum_i a_i\lesssim1.
\end{equation}

It remains to estimate the middle term in \eqref{eq:six-three}. For a fixed pair $(i,j)$ we have:
\[
\begin{aligned}
q^{-1}2^{-i-j}|G_i|^{3/4}|G_j|^{3/4}
&=q^{-1}2^{-i-j}(2^{8i/5}a_i)^{3/4}(2^{8j/5}a_j)^{3/4}=q^{-1}2^{(i+j)/5}a_i^{3/4}a_j^{3/4}\\
&\lesssim q^{-1}2^{(i+j)/5}a_i^{1/2}a_j^{1/2}.
\end{aligned}
\]
This estimate grows with $i+j$, so it cannot be summed directly. For the same pair $(i,j)$, however, the crude Parseval estimate \eqref{eq:six-parseval} gives
\[
2^{-i-j}|\Qcal(h_i,h_j)|\lesssim q2^{-i-j}|G_i|^{1/2}|G_j|^{1/2}.
\]
Again substituting the cardinalities and then simplifying, the right-hand side is
\[
\begin{aligned}
q2^{-i-j}|G_i|^{1/2}|G_j|^{1/2}
&=q2^{-i-j}(2^{8i/5}a_i)^{1/2}(2^{8j/5}a_j)^{1/2}=q2^{-(i+j)/5}a_i^{1/2}a_j^{1/2}.
\end{aligned}
\]
More precisely, for nonnegative quantities $D,S,M,P$ one has
\begin{equation}\label{eq:min-separation-six}
\min\{D+S+M,P\}\le D+S+\min\{M,P\}.
\end{equation}
Apply this with $D$, $S$, and $M$ equal to the diagonal, asymmetric, and middle terms supplied by \eqref{eq:six-bilinear-symmetric}, and with $P$ equal to the Parseval bound. The diagonal and asymmetric terms have already been summed, so the remaining contribution is bounded by
\[
\sum_{i,j\ge0}a_i^{1/2}a_j^{1/2}\min\big\{q^{-1}2^{(i+j)/5},q2^{-(i+j)/5}\big\}.
\]
Using again $a_i^{1/2}a_j^{1/2}\le\frac12(a_i+a_j)$ and the symmetry in $i$ and $j$, this is at most
\[
\sum_{i,j\ge0}a_i\min\big\{q^{-1}2^{(i+j)/5},q2^{-(i+j)/5}\big\}.
\]
For each fixed $i$, the inner sum over $j$ is bounded by the full one-dimensional sum
\[
\sum_{n\ge0}\min\big\{q^{-1}2^{n/5},q2^{-n/5}\big\}.
\]
Let $n_0$ be an integer satisfying $2^{n_0/5}\le q<2^{(n_0+1)/5}$. Then
\[
\begin{aligned}
\sum_{n\ge0}\min\big\{q^{-1}2^{n/5},q2^{-n/5}\big\}
&\le q^{-1}\sum_{0\le n\le n_0}2^{n/5}+q\sum_{n>n_0}2^{-n/5}\\
&\lesssim q^{-1}2^{n_0/5}+q2^{-(n_0+1)/5}\lesssim1.
\end{aligned}
\]
Therefore
\begin{equation}\label{eq:six-middle}
\sum_{i,j\ge0}a_i\min\big\{q^{-1}2^{(i+j)/5},q2^{-(i+j)/5}\big\}\lesssim\sum_i a_i\lesssim1.
\end{equation}

Combining \eqref{eq:six-diag}, \eqref{eq:six-asym}, and \eqref{eq:six-middle} gives $\|\widehat g\|_{L^2(P,d\sigma)}^2\lesssim1$. Since we normalized $\|g\|_{L^{8/5}(F^6)}=1$, this proves $\|\widehat g\|_{L^2(P,d\sigma)}\lesssim\|g\|_{L^{8/5}(F^6)}$. By duality, this is equivalent to $R^*(2\to\frac83)\lesssim1$, which proves Theorem \ref{thm:six-main}.

\section{A transference observation for lattice \texorpdfstring{$\Lambda(p)$}{Lambda(p)} estimates}\label{sec:lambda}

We close with an observation that connects the three-dimensional endpoint conjecture with the analogous problem in lattices. This connection is elementary but, to the author's knowledge, has not been recorded explicitly.

Consider first the integer paraboloid $\Gamma=\{(n_1,n_2,n_1^2+n_2^2):n_1,n_2\in\Z\}\subset\Z^3$, and suppose the endpoint finite field estimate
\begin{equation}\label{eq:ff-endpoint-three}
\|(f d\sigma)^\vee\|_{L^3(\F_p^3)}\lesssim\|f\|_{L^2(P_3,d\sigma)}
\end{equation}
holds uniformly for the three-dimensional paraboloid over all sufficiently large primes $p$ in a congruence class for which $-1$ is not a square. Let $S\subset\Z^2$ be finite, let $(a_n)_{n\in S}$ be complex coefficients, and choose primes $p\to\infty$ so large that reduction modulo $p$ is injective on $\{(n_1,n_2,n_1^2+n_2^2):n\in S\}$. Writing \eqref{eq:ff-endpoint-three} in unnormalized form, and noting that the normalization cancels exactly because $|P_3|=p^2$ while the ambient space has size $p^3$, gives, with $e_p(t)=e^{2\pi it/p}$ the standard additive character of $\F_p$,
\[
\Big(p^{-3}\sum_{x\in\F_p^3}\Big|\sum_{n\in S}a_n e_p(x_1n_1+x_2n_2+x_3(n_1^2+n_2^2))\Big|^3\Big)^{1/3}\lesssim\Big(\sum_{n\in S}|a_n|^2\Big)^{1/2}.
\]
As $p\to\infty$, the normalized sum inside the cube root is a Riemann sum for the corresponding integral over $\T^3=(\mathbb R/\Z)^3$. Thus the left-hand side converges to the $L^3(\T^3)$ norm of the trigonometric sum, in which $e(t)=e^{2\pi it}$, and the finite field endpoint would imply
\begin{equation}\label{eq:integer-paraboloid-lambda3}
\Big\|\sum_{n\in S}a_n e(n_1t_1+n_2t_2+(n_1^2+n_2^2)t_3)\Big\|_{L^3(\T^3)}\lesssim\Big(\sum_{n\in S}|a_n|^2\Big)^{1/2},
\end{equation}
with a constant independent of the finite set $S$. Under this hypothesis, the extension operator initially defined on finitely supported coefficient sequences on $\Gamma$ is bounded from $\ell^2(\Gamma)$ to $L^3(\T^3)$. By the bounded linear extension theorem it extends to all of $\ell^2(\Gamma)$, and Plancherel identifies the extension with the corresponding Fourier series. Thus the full integer paraboloid $\Gamma$ would be a $\Lambda(3)$ set.

The same argument applies to spheres, with the appropriate restriction on the finite field radius. Let $S_R=\{n\in\Z^3:n_1^2+n_2^2+n_3^2=R\}$, where $R$ is a nonnegative integer, and let $\chi_p$ denote the quadratic character of $\F_p$. Suppose that the $L^2\to L^3$ extension estimate holds for the spheres $\{x\in\F_p^3:x\cdot x=r\}$ with $\chi_p(-r)=-1$, with a constant uniform in the odd prime $p$ and the radius $r$. These are precisely the nonzero-radius spheres containing no affine lines. If $R=0$ the lattice sphere is a singleton. For each fixed $R>0$, there are arbitrarily large primes $p\nmid 2R$ with $\chi_p(-R)=-1$, since $-R$ is not a square in $\mathbb Q$. Along such primes, reduction is eventually injective on $S_R$, and the same normalization and limiting argument gives
\[
\Big\|\sum_{n\in S_R}a_n e(n\cdot t)\Big\|_{L^3(\T^3)}\lesssim\Big(\sum_{n\in S_R}|a_n|^2\Big)^{1/2},
\]
uniformly in the radius and in the finite support. Thus each fixed lattice sphere $S_R$ would be a $\Lambda(3)$ set, with a constant independent of $R$.

The normalization explains the role of the exponent $3$. If $V\subset\F_q^d$ has size $\sim q^{d-1}$ and $R^*(2\to r)\lesssim1$, then the corresponding normalized exponential-sum estimate, with $e_q$ a fixed nontrivial additive character of $\F_q$, reads
\begin{equation}\label{eq:general-lambda-loss}
\Big(q^{-d}\sum_{x\in\F_q^d}\Big|\sum_{\xi\in V}a_\xi e_q(x\cdot\xi)\Big|^r\Big)^{1/r}\lesssim q^{(d-1)/2-d/r}\Big(\sum_{\xi\in V}|a_\xi|^2\Big)^{1/2},
\end{equation}
with no power of $q$ exactly when $r=\frac{2d}{d-1}$, which for $d=3$ is the conjectured finite field endpoint $r=3$. For paraboloids and spheres in ambient dimensions $d\ge4$, affine subspaces force the finite field $L^2$ endpoint above $\frac{2d}{d-1}$, so this normalization leaves a power of $q$.

The conditional $\Lambda(3)$ estimate falls short of the full conjecture. The discrete restriction conjecture of Bourgain \cite{BouLattice} predicts that the integer paraboloid is a $\Lambda(r)$ set for every $r<4$, and that the lattice spheres satisfy uniform $\Lambda(r)$ estimates for every $r<6$, the wider range for the sphere reflecting the greater arithmetic sparsity of the discrete sphere \cite{Bou,BouSurvey}. For the paraboloid the exponent $r=4$ is critical, so the conditional estimate is subcritical. At frequency scale $N$, an $N^\eps$ loss in the critical $r=4$ estimate was obtained in \cite{BouLattice}, and follows on arbitrary tori from the $\ell^2$ decoupling theorem of Bourgain and Demeter \cite{BD}. Herr and Kwak \cite[Theorem 1.1]{HK} obtained the sharp loss $(\log(2+|S|))^{1/4}$ at $r=4$, where $S\subset\Z^2$ is the finite frequency support. The analogous $r=6$ estimate for lattice spheres with an $N^\eps$ loss remains open; see \cite{BouSurvey}.

\noindent\emph{Note added.} After this paper was written, Deng, Fan, S.~Guo, Z.~Guo and Luo \cite[Theorem 1.2]{DFFGL} showed that the one-dimensional integer parabola $\Gamma_1=\{(n,n^2):n\in\Z\}\subset\Z^2$ is a $\Lambda(r)$ set for every $2<r<6$, with the sharp loss $(\log|S|)^{1/6}$ at the endpoint $r=6$. It is not a $\Lambda(6)$ set, by Bourgain's logarithmic lower bound \cite{BouLattice}.

\section*{Statement on AI usage}

All mathematical arguments and conclusions in this paper are the author's own, and the author takes full responsibility for their correctness. As noted in the footnote in the introduction, the basic approach was already known to the author when the earlier paper was written and was mentioned there, but it was not pursued because of confusion over the dominant term. The author used ChatGPT as a sounding board while organizing the optimization and checking the cutoff calculations, and used ChatGPT, Gemini, and Claude as editorial and technical aids for LaTeX, formatting, copyediting, and proofreading.

\end{document}